\documentclass{acm_proc_article-sp}
\usepackage[all]{xy}

\def\N{\mathbb{N}}

\def\R{\mathbb{R}}

\def\C{\mathbb{C}}

\def\proj{\mathbb{P}}

\renewcommand{\a}{\alpha}

\renewcommand{\d}{\delta}

\newcommand{\g}{\gamma}

\newcommand{\Om}{\Omega}
\newcommand{\om}{\omega}
\newcommand{\z}{\zeta}

\def\mZ{{\mathcal Z}}

\newcommand{\x}{\times}

\renewcommand{\hat}{\widehat}

\renewcommand{\tilde}{\widetilde}

\def\Oh{{\cal O}}

\def\algorithm{\begin{center}
               \begin{minipage}{6in}
               \begin{tabbing}
               \marks}
\def\falgorithm{\end{tabbing}
                \end{minipage}
                \end{center}}
\def\marks{nn\= nn\= nn\= nn\= nn\= nn\= nn\= \kill}

\def\P{{\sf P}}

\def\PSPACE{{\sf PSPACE}}

\def\PR{{\rm P}_{\kern-1pt\R}}

\def\PC{{\rm P}_{\kern-1pt\C}}
\def\NPR{{\rm NP}_{\kern-1pt\R}}
\def\NPC{{\rm NP}_{\kern-2pt\C}}
\def\DNPR{{\rm DNP}_{\kern-1pt\R}}
\def\DNPC{{\rm DNP}_{\kern-2pt\C}}
\def\PAR{{\rm PAR}_{\kern-1pt\R}}
\def\PHR{{\rm PH}_{\kern-1pt\R}}
\def\DPHR{{\rm DPH}_{\kern-1pt\R}}

\def\FPR{{\rm FP}_{\kern-1pt\R}}
\def\FPC{{\rm FP}_{\kern-1pt\C}}
\def\FPAR{{\rm FPAR}_{\kern-0.4pt\R}}
\def\FPARC{{\rm FPAR}_{\kern-0.4pt\C}}

\def\CPRi{{\rm \#P}_{\kern-2pt\R}}
\def\CPCi{{\rm \#P}_{\kern-2pt\C}}

\def\FEASR{{\mbox{\sc Feas}_{\kern-0.5pt\R}}}  
\def\FEASRbit{{\mbox{\sc Feas}^0_{\kern-1pt\R}}}
\def\SAS{{\mbox{\sc SAS}_{\kern-0.5pt\R}}}
\def\SASbit{{\mbox{\sc SAS}_{\kern-1pt\R}^0}}
\def\HNC{{\mbox{\sc HN}_{\kern-1pt\C}}}
\def\HNCbit{{\mbox{\sc HN}^0_{\kern-1pt\C}}}
\def\QASC{{\mbox{\sc QAS}_{\kern-1pt\C}}}

\def\DIMR{{\mbox{\sc Dim}_{\kern-0.5pt\R}}}
\def\DIMC{{\mbox{\sc Dim}_{\kern-0.5pt\C}}}
\def\DIMadd{{\mbox{\sc Dim}_{\kern-0.5pt\add}}}
\def\DIMRbit{{\mbox{\sc Dim}^0_{\kern-0.5pt\R}}}
\def\DIMCbit{{\mbox{\sc Dim}^0_{\kern-0.5pt\C}}}
\def\REACH{{\mbox{\sc Reach}_{\kern-0.5pt\R}}}
\def\REACHbit{{\mbox{\sc Reach}^0_{\kern-0.5pt\R}}}
\def\CREACHbit{{\mbox{\sc CReach}^0_{\kern-0.5pt\R}}}

\def\REACH{{\mbox{\sc Reach}_{\kern-0.5pt\R}}}
\def\REACHbit{{\mbox{\sc Reach}^0_{\kern-0.5pt\R}}}
\def\CREACHbit{{\mbox{\sc CReach}^0_{\kern-0.5pt\R}}}

\newcommand{\HHNC}{\mbox{\sc HN}_{\kern-1pt\C}}

\newcommand{\ud}{\mathrm{d}}

\def\id{\mathrm{id}}

\def\FPk{{\rm FP}_{\kern-1pt k}}
\def\Pk{{\rm P}_{\kern-1pt k}}
\def\NPk{{\rm NP}_{\kern-2pt k}}

\newcommand\CPpar[1]{{\rm \#P}_{\kern-2pt #1}}

\def\HNk{\textsc{HN}_{\kern-1pt k}}
\def\DIMk{{\mbox{\sc Dim}_{\kern-0.5pt k}}}

\newcommand{\comment}[1]{}

\def\Hdr{H_{\rm{dR}}}

\def\Res{{\rm Res}}

\newtheorem{lemma}{Lemma}
\newtheorem{theorem}[lemma]{Theorem}

\newtheorem{corollary}[lemma]{Corollary}

\newtheorem{example}[lemma]{Example}

\begin{document}
\title{Effective de Rham Cohomology - The Hypersurface Case}
\numberofauthors{1}
\author{
\alignauthor
Peter Scheiblechner\\
       \affaddr{Hausdorff Center for Mathematics}\\
       \affaddr{Endenicher Allee 62}\\
       \affaddr{53115 Bonn, Germany}\\
       \email{peter.scheiblechner@hcm.uni-bonn.de}
}

\maketitle

\begin{abstract}
We prove an effective bound for the degrees of generators of the algebraic de
Rham cohomology of smooth affine hypersurfaces. In particular, we show that
the de Rham cohomology $\Hdr^p(X)$ of a smooth hypersurface $X$ of degree~$d$
in~$\C^n$ can be generated by differential forms of degree $d^{\Oh(pn)}$.
This result is relevant for the algorithmic computation of the cohomology, but
is also motivated by questions in the theory of ordinary differential equations related
to the infinitesimal Hilbert 16th problem.
\end{abstract}




\section{Introduction}
Let $X$ be a smooth variety in $\C^n$.
A fundamental result of Grothendieck says that the cohomology of $X$ can be
described in terms of \textit{algebraic} differential forms on~$X$~\cite{gro:66}.
More precisely, he proved that the singular cohomology of $X$ is isomorphic to
the algebraic de Rham cohomology $\Hdr^\bullet(X)$, which is defined as the
cohomology of the complex of algebraic differential forms on $X$.
Hence, each cohomology class in $\Hdr^p(X)$ can be represented by a $p$-form
\begin{equation}\label{eq:diffForm}
\omega=\sum_{i_1<\cdots<i_p}\omega_{i_1\cdots i_p}\ud X_{i_1}\wedge\cdots\wedge \ud X_{i_p},
\end{equation}
where the $\omega_{i_1\cdots i_p}$ are polynomial functions on $X$. However,
Grothendieck's proof gives no information on the degrees of the polynomials
$\omega_{i_1\cdots i_p}$.
In this paper we prove a bound on their degrees in the case of a hypersurface.
In a forthcoming paper, we will consider the general case.

\subsection{Motivation}
It is a long standing open question in algorithmic real algebraic geometry to
find a single exponential time algorithm for computing the Betti numbers of a
semialgebraic set. Single exponential time algorithms are known, e.g., for
counting the connected components and computing the Euler characteristic of a
semialgebraic set (for an overview see~\cite{basu:08}, for details and
exhaustive bibliography see~\cite{bpr:03}). The best result in this direction
states that for fixed $\ell$ one can compute the first $\ell$ Betti numbers of
a semialgebraic set in single exponential time~\cite{basu:06}.

Over the complex numbers, one approach for computing Betti numbers is to 
compute the algebraic de Rham cohomology. In~\cite{ota:98,wal:00} the de Rham
cohomology of the complement of a complex affine variety is computed using
Gr\"{o}bner bases for $\mathcal{D}$-modules. This algorithm is extended in
\cite{wal:00b} to compute the cohomology of a projective variety.
However, the complexity of these algorithms is not analyzed, and due to their
use of Gr\"{o}bner bases a good worst-case complexity is not to be expected.
In~\cite{bus:09} a single exponential time (in fact, parallel polynomial
time) algorithm is given for counting the connected components, i.e., computing
the zeroth de Rham cohomology, of a (possibly singular) complex variety.
This algorithm is extended in~\cite{sch:10} to one with the same complexity
for computing equations for the components.
The first single exponential time algorithm for computing \textit{all} Betti
numbers of an interesting class of varieties is given in~\cite{sch:09}. Namely,
this paper shows how to compute the de Rham cohomology of a smooth projective
variety in parallel polynomial time. In terms of structural complexity, these
results are the best one can hope for, since the problem of computing a fixed
Betti number (e.g., deciding connectedness) of a complex affine or projective
variety defined over the integers is $\PSPACE$-hard~\cite{sch:07}.

Besides being relevant for algorithms, our question also has connections to the
theory of ordinary differential equations. 
The long standing infinitesimal Hilbert 16th problem has been solved
in~\cite{bny:10}.
The authors derive a bound on the number of limit cycles generated from
nonsingular energy level ovals (isolated periodic trajectories) in a
non-conservative perturbation of a Hamiltonian polynomial vector field in the
plane.
It seems that their proof can be considerably generalized to solutions of
certain linear systems of Pfaffian differential equations. Examples of such
systems are provided by period matrices of polynomial maps, once the
corresponding Gauss-Manin connexion can be explicitly constructed. For this
construction one needs degree bounds for generators of the cohomology of the
generic fibers of the polynomial map.

\subsection{Known Cases}\label{ss:knownCases}
It follows from the results of~\cite{sch:09} that if $X$ has no singularities
at infinity, i.e., the projective closure of $X$ in $\proj^n$ is smooth, then
each class in $\Hdr^\bullet(X)$ can be represented by a differential form of
degree at most $m(em+1)d$, where $m=\dim X$, $d=\deg X$, and $e$ is the maximal
codimension of the irreducible components of $X$. However, in general $X$ does
have singularities at infinity, and resolution of singularities has a very bad
worst-case complexity~\cite{bgmw:11}.

Another special case with known degree bounds is the complement of a projective
hypersurface, which we will actually use in this paper (see the proof of
Theorem~\ref{thm:effDeRham}). The statement follows
essentially from~\cite{ded:90} and~\cite{dimc:90}, the argument can be found
in~\cite[Corollary 6.1.32]{dimc:92}. Let $f\in\C[X_0,\ldots,X_n]$ be a
homogeneous polynomial, and consider $U:=\proj^n\setminus\mZ(f)$, which is an
affine variety. Then, each class in $\Hdr^p(U)$ is represented by a
(homogeneous) differential form
$$
\frac{\a}{f^p}\quad\text{with}\quad\deg \a=p\deg f
$$
(see Section~\ref{sec:basics} for the definition of the degree of a differential
form).

\subsection{Main Result}
In this paper we prove that each class in the de Rham cohomology~$\Hdr^p(X)$ of
a smooth hypersurface $X\subseteq\C^n$ of degree~$d$ can be represented by a
differential $p$-form of degree at most
$$
(p+1)(d+1)(2d^n+d)^{p+1}\le d^{\Oh(pn)}
$$
(see Theorem~\ref{thm:effDeRham}).

\section{Preliminaries}\label{sec:basics}

An (affine) \textit{variety} in $\C^n$ is the common zero-set
$$
X=\mZ(f_1,\ldots,f_r)=\{x\in\C^n\,|\,f_1(x)=\cdots=f_r(x)=0\}
$$
of a set of polynomials $f_1,\ldots,f_r\in\C[X_1,\ldots,X_n]$. 
The \textit{coordinate ring} of a variety $X$ is $\C[X]=\C[X_1,\ldots,X_n]/I(X)$,
where $I(X):=\{f\in \C[X_1,\ldots,X_n]\,|\,f(x)=0\ \forall x\in X\}$ is the
\textit{vanishing ideal} of $X$. Such a $\C$-algebra is called a
\textit{(reduced) affine algebra}. By Hilbert's Nullstellensatz, $I(X)$ is the
radical of the ideal generated by $f_1,\ldots,f_r$. In particular, if $X$ is a
hypersurface defined by one polynomial~$f$, then $I(X)=(g)$, where $g$ is the
squarefree part of $f$, which is also called a \textit{reduced} equation of
$X$. A variety $X$ (and then also its coordinate ring) is called \textit{smooth},
if at any point $x\in X$ we have
$\dim_x X=\dim T_xX$, i.e., at $x$ the local dimension and the dimension of the
tangent space of~$X$ coincide. In particular, a hypersurface is smooth if and
only if $\mZ(f,\partial_1f,\ldots,\partial_1f)=\emptyset$, where $f$ is a
reduced equation of~$X$.

We will also use complete rings, which we get from affine rings by the process
of completion~\cite[Chapter 7]{eise:95}. Let $A$ be an affine algebra and $I$
an ideal in $A$. The \textit{completion} $\hat{A}=\hat{A}_I$ of~$A$ with
respect to $I$ is defined as the inverse limit of the factor rings
$A/I^\nu$, $\nu\ge 0$. There is a canonical map $A\to\hat{A}$, whose kernel is
$\bigcap_\nu I^\nu$, so it is injective in our case.
Alternatively, if $I=(f_1,\ldots,f_r)$, one can define $\hat{A}$ as
$A[[T_1,\ldots,T_r]]/(T_1-f_1,\ldots,T_r-f_r)$, so its elements
are power series in $f_1,\ldots,f_r$~\cite[Exercise 7.11]{eise:95}.
For instance, if $A=B[T]$ and $I=(T)$, then $\hat{A}=B[[T]]$ is the ring of
formal power series in $T$ with coefficients in $B$.

Let $A$ be a $\C$-algebra (commutative, with 1).
The module of \textit{K\"{a}hler differentials} $\Om_A:=\Om_{A/\C}$ is defined
as the the $A$-module generated by symbols $\ud f$ for all $f\in A$, modulo
the relations of Leibniz' rule and $\C$-linearity for the
\textit{universal derivation} $\ud\colon A\to\Om_A$. For instance, if $A=\C[X_1,\ldots,X_n]$,
then $\Om_A$ is the free module with basis $\ud X_1,\ldots,\ud X_n$ and
$\ud f=\sum_i\partial_i f\ud X_i$.

Now let $\Om_A^p:=\bigwedge^p\Om_A$ be the $p$-th exterior power of the
$A$-module $\Om_A$.
We define the differential $\ud\colon \Om^p_A\to\Om^{p+1}_A$ by setting
$\ud (f \ud g_1\wedge\cdots\wedge\ud g_p):=\ud f\wedge\ud g_1\wedge\cdots\wedge\ud g_p$ for
$f,g_i\in A$.
It is easy to check that~$\ud$ satisfies the graded Leibniz' rule and
$\ud\circ\ud=0$. This  way we obtain the \textit{(algebraic) de Rham complex}
$$
\Om^\bullet_A\colon A=\Om^0_A\stackrel{\ud}{\longrightarrow}\Om_A^1\stackrel{\ud}
{\longrightarrow}\cdots\stackrel{\ud}{\longrightarrow}\Om_A^n\longrightarrow 0.
$$
Its cohomology $\Hdr^\bullet(A):=H^\bullet(\Om^\bullet_A)$ is called the
\textit{de Rham cohomology} of $A$. If $A=\C[X]$ is the coordinate ring of a
smooth variety $X\subseteq\C^n$, then $\Hdr^\bullet(X):=\Hdr^\bullet(A)$ is called
the \textit{(algebraic) de Rham cohomology} of $X$. Fundamental for us is the
result of~\cite{gro:66} stating that the de Rham cohomology $\Hdr^\bullet(X)$ of
a smooth variety $X$ is naturally isomorphic to the singular cohomology of $X$.

The module of K{\"a}hler differentials of a complete ring may not be finitely
generated (see, e.g.,~\cite[Exercise 16.14]{eise:95}).
In these cases we use the \textit{universally finite} module of differentials,
which is always finitely generated (see~\cite[\S11--12]{kun:86}).
Let $R$ be an affine algebra, $I$ an ideal in $R$, and $\hat{R}$ the
completion of $R$ with respect to $I$. The \textit{completion} of~$\Om_{R}$ with
respect to $I$ is $\hat{\Om}_{\hat{R}}=\hat{R}\otimes_R\Om_R$ and is called the
\textit{universally finite module of differentials of}~$\hat{R}$. There is a
\textit{universally finite derivation} $\ud\colon\hat{R}\to \hat{\Om}_{\hat{R}}$
which is continuous, i.e., it commutes with infinite sums. For instance, for an
affine algebra $B$ we have
$$
\hat{\Om}_{B[[T]]}=B[[T]]\otimes\Om_{B[T]}=B[[T]]\ud T\oplus\Om_{B}
$$
and $\ud f=\partial_T f\ud T+\ud_B f$ for $f\in B[[T]]$,
where $\partial_T$ denotes partial derivative with respect to $T$ and
$\ud_B$ denotes the derivation of $\Om_B$ applied to the coefficients of $f$~\cite[Example 12.7]{kun:86}. 

Now let $A=\C[X_1,\ldots,X_n]/I$ be an affine algebra. We are interested in
degree bounds for the de
Rham cohomology of $A$, so we introduce the following notation. For
$f\in\C[X_1,\ldots,X_n]$ we denote by~$\overline{f}$ its residue class in $A$.
We set
$$
\deg\overline{f}:=\min\{\deg g\,|\,g\in\C[X_1,\ldots,X_n],\ \overline{g}=\overline{f}\}.
$$
Furthermore, we define the degree of differential forms in~$\Om_A^p$ by setting
$$
\deg(\overline{f}\ud \overline{X}_{i_1}\wedge\cdots\wedge\overline{X}_{i_p}):=
\deg\overline{f}+p.
$$
We denote by $\deg(\Hdr^p(A))$ the infimum over all integers $\d$ such that each
cohomology class in $\Hdr^p(A)$ has a representative of degree $\le\d$.
For a localization $A_f$, we define
$$
\deg\frac{a}{f^s}:=\deg\a-s\deg f,\quad \a\in\Om_A^p.
$$

\section{Effective Gysin Sequence}
The main tool in our proof is the Gysin sequence which is the following
\begin{lemma}\label{lem:gysinSequ}
Let $Y$ be an irreducible smooth variety and $X\subseteq Y$ a smooth
hypersurface. Then there is an exact sequence
$$
\cdots\rightarrow \Hdr^p(Y)\to \Hdr^p(Y\setminus X)\stackrel{\Res}{\to}
\Hdr^{p-1}(X)\to \Hdr^{p+1}(Y)\to\cdots
$$
\end{lemma}
A proof of this Lemma is sketched in~\cite{mon:68}. Along the same lines we
will prove an effective version of it for the case $Y=\C^n$. This means
that we will describe the map Res explicitly on the level of differential forms,
so that we can control its effect on their degree.
Let us first record an easy consequence of the Gysin sequence.
Since $\Hdr^p(\C^n)=0$ for $p>0$, Lemma~\ref{lem:gysinSequ} implies
\begin{corollary}
For a smooth hypersurface $X\subseteq\C^n$ the residue map
$$
\Res\colon\Hdr^p(\C^n\setminus X)\stackrel{\simeq}{\longrightarrow} \Hdr^{p-1}(X)
$$
is an isomorphism for all $p>0$.
\end{corollary}
Let $A:=\C[X_1,\ldots,X_n]$ and $X:=\mZ(f)$ smooth, where $f\in A$ is squarefree.
Then the relevant coordinate rings are
$$
\C[\C^n\setminus X]=A_f \quad\text{and}\quad B:=\C[X]=A/(f).
$$
Furthermore, $d:=\deg X=\deg f$. We assume $d\ge 3$.
\begin{theorem}\label{thm:effGysin}
Let $d\ge 3, p>0$. The residue map
$$
\Res\colon\Hdr^p(A_f)\to\Hdr^{p-1}(B)
$$
is induced by a map $\Om_{A_f}^p\to\Om_B^{p-1}$ which takes a $p$-form
$\om=\frac{\a}{f^s}$ to a $(p-1)$-form of degree at most
$$
(2d^n+d)^{s}(\deg \om+sd).
$$
\end{theorem}
\begin{example}\label{ex:univar}
It is instructive to consider the case $n=1$. Consider $f=\prod_i^d(X-\z_i)\in\C[X]$,
and let $\om=\frac{g}{f^s}\ud X\in\Om_{\C[X]_f}$. 
By expanding the rational function $\frac{g}{f^s}$ into partial fractions and
noting that polynomials and the terms of the form $c(X-\z_i)^j$ with $j<-1$ can
be integrated, one sees that the cohomology $\Hdr^1(\C\setminus\mZ(f))$ is
generated by the differential forms
$$
\frac{\ud X}{X-\z_i},\quad 1\le i\le d.
$$
Since they are also linearly independent, they form a basis of this cohomology.
We will see later that the residue map sends $\frac{\ud X}{X-\z_i}$ to
$e_i:=\prod_{j\ne i}(X-\z_j)/a_i$, where $a_i:=\prod_{j\ne i}(\z_i-\z_j)$. 
Since $\mZ(f)=\{\z_1,\ldots,\z_d\}$ and $e_i(\z_j)=\d_{ij}$,
the $e_i$ are the idempotents corresponding to the points $\z_i$ and thus are a
basis of the cohomology $\Hdr^0(\mZ(f))$.
\end{example}
For the proof of Theorem~\ref{thm:effGysin} we will need the completion
$\hat{A}=\hat{A}_{(f)}$ of the algebra~$A$ with respect to the principal ideal
$(f)$. Recall from \S\ref{sec:basics} that $\hat{A}=A[[T]]/(T-f)$, so its
elements are power series in $f$.
Note, however, that these power series are not unique. E.g., $f\in A$ can be
represented as $f+0\cdot f+\cdots$ or $0+1\cdot f+0\cdot f^2+\cdots$. The
crucial result for us is a lemma of Grothendieck stating that there is a ring
isomorphism $B[[T]]\to\hat{A}$ (cf.~\cite[Lemma II,1.2]{har:75}), which
establishes a unique
power series representation for the completion. We need to construct this
isomorphism explicitly in order to bound degrees. It is easy to come up with a
vector space isomorphism, so the difficulty lies in finding a ring isomorphism.
The technical construction is in the following statement, which expresses
that~$B$ is a \textit{formally smooth}
$\C$-algebra~\cite[Definition 19.3.1]{gro:64}. 

For a tuple $x=(x_1,\ldots,x_n)$ over an affine algebra $R$ we write
$\deg(x):=\max_j\deg(x_j)$. If $\psi\colon R\to S$ is a homomorphism, we write
$\psi(x):=(\psi(x_1),\ldots,\psi(x_n))$.
For $g\in R$ we denote by $\overline{g}$ its image in any factor algebra of $R$.

\begin{lemma}\label{lem:groth}
Let $A=\C[X_1,\ldots,X_n]$ and $f\in A$ be squarefree such that $B=A/(f)$ is
smooth.
Furthermore, let $\nu\in\N$ and $\psi\colon B\to A/(f^\nu)$ be a ring
homomorphism that lifts
the identity $B\to B$, i.e., the composition $B\to A/(f^\nu)\twoheadrightarrow B$
is the identity. Then $\psi$ can be lifted to a ring homomorphism
$\tilde{\psi}\colon B\to A/(f^{\nu+1})$, i.e., the diagram
$$
\xymatrix{
 & A/(f^{\nu+1}) \ar@{->>}[d]^\pi \\
B \ar[r]^\psi\ar@{-->}[ru]^{\tilde{\psi}} &  A/(f^\nu) \\
}
$$
commutes. Furthermore, we have
$$
\deg(\tilde{\psi}(\overline{X}))\le d\cdot \deg(\psi(\overline{X}))+d^n.
$$
\end{lemma}
\begin{proof}
Since $B$ is generated as a $\C$-algebra by the $\overline{X}_i$,
it suffices to define $\tilde{\psi}$ on these elements.
Choose $Y_1,\ldots,Y_n\in A$ such that $\psi(\overline{X}_i)=\overline{Y}_i$
in $A/(f^{\nu})$ for all $1\le i\le n$. Our aim is to define
\begin{equation}\label{eq:defPsi}
\tilde{\psi}(\overline{X}_i):=\overline{Y_i+a_i f^\nu},\quad 1\le i\le n,
\end{equation}
with suitably chosen $a_i\in A$. Then it is clear that
$\pi\circ\tilde{\psi}=\psi$. It remains to show that one can define
$\tilde{\psi}$ unambigously by~\eqref{eq:defPsi}. This means that we have to
find~$a_i$ such that $f$ is mapped to zero in $A/(f^{\nu+1})$. Set
$Y:=(Y_1,\ldots,Y_n)$, $a:=(a_1,\ldots,a_n)$, and look at the condition
\begin{equation}\label{eq:cond}
f\mapsto \overline{f(Y+a f^\nu)}=0\quad \text{in}\quad A/(f^{\nu+1})
\end{equation}
By the Taylor formula we have
$$
f(Y+a f^\nu)\equiv f(Y)+\sum_{i=1}^n\partial_i f(Y)a_if^\nu\pmod{(f^{\nu+1})}.
$$
Since $\overline{f(Y)}=\psi(\overline{f})=0$ in $A/(f^{\nu})$, there exists
$p\in A$ such that $f(Y)=p f^\nu$ in $A$.
Furthermore, since $\overline{Y}_i=\overline{X}_i$ in $B$, condition~\eqref{eq:cond}
is satisfied if
\begin{equation}\label{eq:condLGS}
p+\sum_{i=1}^n\partial_i fa_i\equiv 0\pmod{(f)}.
\end{equation}

Since $\mZ(f)$ is smooth, we have $\mZ(f,\partial_1 f,\ldots,\partial_n f)=\emptyset$,
so by the effective Nullstellensatz~\cite{brow:87,koll:88,jel:05} there exist $g_1,\ldots,g_n\in A$ such that
\begin{equation}\label{eq:smoothnessHN}
\sum_{i=1}^n g_i\partial_i f\equiv 1\pmod{(f)}\quad \text{and}
\quad \deg g_i\le d^n.
\end{equation}
It follows that~\eqref{eq:condLGS} can be solved by choosing
$$
a_i:=-p g_i\quad \text{for all}\quad 1\le i\le n.
$$
Furthermore,
$$
\deg a_i\le\deg p +\deg g_i \le d\deg Y-\nu d+ d^n,
$$
which implies the claimed degree bound.
\end{proof}

\begin{corollary}\label{cor:psi}
There exists an embedding $\psi\colon B\hookrightarrow\hat{A}$ such that
$\psi(\overline{X}_i)=\sum_{\nu=0}^\infty a_{i\nu}f^\nu$, where $a_{i\nu}\in A$ with
$$
\deg a_{i\nu}\le d^n\sum_{i=0}^{\nu-1}d^i+d^\nu\le d^\nu(2d^{n-1}+1).
$$
\end{corollary}
\begin{proof}
We start with $\psi_1:=\id_B$ and apply Lemma~\ref{lem:groth} successively to
construct the homomorphisms $\psi_\nu\colon B\to A/(f^\nu)$, $\nu\in\N$.
Together they define a homomorphism $\psi\colon B\to\hat{A}$, which is clearly
injective.
For $\nu=0$ we have $\psi_0(\overline{X})=\overline{X}$, whose degree is 1. By
Lemma~\ref{lem:groth} we have
\begin{align*}
\deg a_{i,\nu+1} & \le d \deg a_{i,\nu}+d^n\le d(d^n\sum_{i=0}^{\nu-1}d^i+d^\nu)+d^n\\
 & = d^n\sum_{i=0}^{\nu}d^i+d^{\nu+1}.\qed
\end{align*}
\end{proof}

For $a\in \hat{A}$ we write $\deg_\nu a\le \d_\nu$, if there exists a
representation $a=\sum_\nu a_\nu f^\nu$ with $\deg a_\nu\le \d_\nu$ for all
$\nu\in\N$. The degree bound from the previous corollary reads in this notation
$\deg_{\nu}(\psi(\overline{X}_i))\le d^\nu(2d^{n-1}+1)$.
\begin{corollary}\label{cor:directSum}
For all $a=\sum_\nu a_\nu f^\nu\in \hat{A}$ there exist unique $b\in B$ and
$c\in\hat{A}$ such that $a=\psi(b)+cf$ and
$$
\deg b\le\deg a_0,
$$
$$
\quad\deg_\nu c\le \max\{\deg a_{\nu+1},d^{\nu+1}(2d^{n-1}+1)\deg a_0\}.
$$
\end{corollary}
\begin{proof}
We have the exact sequence of $\hat{A}$-modules
$$
0\longrightarrow (f)\longrightarrow\hat{A}\stackrel{\pi}{\longrightarrow} B\longrightarrow 0,
$$
which splits by the homomorphism $\psi$. Hence, $\hat{A}\simeq B\oplus (f)$,
and the existence and uniqueness of the claimed representation follows. Note
that if $a=\psi(b)+cf$, then $b=\pi(a)=\pi(a_0)$. This implies the bound for
$b$.
Since $\psi$ and $\pi$ are homomorphisms and $a_0$ is a polynomial, we have
$$
\psi(b)=\psi(\pi(a_0))=a_0(\psi(\overline{X}))=a_0(\xi),
$$
where we denote $\xi_i:=\psi(\overline{X}_i)$ and $\xi=(\xi_1,\ldots,\xi_n)$. 
Hence, $cf=a-a_0(\xi)$ and $\deg_\nu c\le\max\{\deg_{\nu+1}a,\deg_{\nu+1} a_0(\xi)\}$.
Thus, it remains to bound the degree of $a_0(\xi)$. We can assume that $a_0$ is
a monomial, and for that a straight-forward induction with respect to the
degree using Corollary~\ref{cor:psi} shows 
$$
\deg_\nu a_0(\xi)\le d^{\nu}(2d^{n-1}+1)\deg a_0,
$$
which implies the claimed bound for the degrees.
\end{proof}

Now we define the homomorphism
$$
\hat{\psi}\colon B[[T]]\to\hat{A},\quad\sum_\nu b_\nu T^\nu\mapsto\sum_\nu \psi(b_\nu) f^\nu.
$$
\begin{lemma}\label{lem:iso}
The homomorphism $\hat{\psi}$ is an isomorphism, and 
we have
$$
\deg_\nu\hat{\psi}^{-1}(a) \le (2d^n+d)^\nu\deg a 
$$
for all $a\in A$.
\end{lemma}
\begin{proof}
The injectivity of $\psi$ implies inductively that $\hat{\psi}$ is injective.
To show surjectivity,  let $a\in\hat{A}$. Construct $\sum_\nu b_\nu T^\nu\in B[[T]]$
with $\sum_\nu\psi(b_\nu)f^\nu=a$. We find the $b_\nu$ successively by applying
Corollary~\ref{cor:directSum}. Let $b_0\in B$ and $c_0\in \hat{A}$ with
$a=\psi(b_0)+c_0 f$. Then, there is  $b_1\in B$ and $c_1\in \hat{A}$ with
$c_0=\psi(b_1)+c_1 f$, and so forth. It follows
$a=\psi(b_0)+c_0 f=\psi(b_0)+\psi(b_1)f+c_1 f^2=\cdots=\sum_\nu \psi(b_\nu) f^\nu$,
which is the image of $\sum_\nu b_\nu T^\nu$. With $\g:=2d^{n-1}+1$ we first prove
$$
\deg_\nu c_\mu\le d^{\nu+\mu+1}\g^{\mu+1}\deg a,
\quad\nu,\mu\ge 0,
$$
by induction on $\mu$. For $\mu=0$ the claim follows directly
from Corollary~\ref{cor:directSum}. For $\mu\ge0$, we have by
induction hypothesis
\begin{align*}
\deg_\nu c_{\mu+1} \le & \max\big\{d^{\nu+\mu+2}\g^{\mu+1}\deg a,
 d^{\nu+1}\g d^{\mu+1}\g^{\mu+1}\deg a\big\}\\
\le & d^{\nu+\mu+2}\g^{\mu+2}\deg a,
\end{align*}
which proves the claim. Now, again by Corollary~\ref{cor:directSum}
$$
\deg b_\mu \le \deg_0 c_{\mu-1} \le (d\g)^\mu\deg a.\qed
$$
\end{proof}

\begin{proof}[of Theorem~\ref{thm:effGysin}]
Consider the short exact se\-quen\-ce of complexes
$$
0\to \Om^\bullet_A\to\Om^\bullet_{A_f}\to \Om^\bullet_{A_f}/\Om^\bullet_{A}\to 0.
$$
It induces a long exact sequence
\begin{align*}
\cdots & \to H^p(\Om^\bullet_A) \to H^p(\Om^\bullet_{A_f})\to
H^p(\Om^\bullet_{A_f}/\Om^\bullet_{A})\to \\
 & \to H^{p+1}(\Om^\bullet_A)\to \cdots.
\end{align*}
Let $p\ge 1$.
We have a map $\lambda\colon\Om_B^{p-1}\to \Om^p_{A_f}/\Om^p_{A}$ which sends
the class of $\om$ to $\frac{\ud f}{f}\wedge\om'$, where $\om'\in\Om_A$ is a
lift of $\om$.
The residue map $\Res\colon\Hdr^p(A_f)\to \Hdr^{p-1}(B)$ is a cohomology inverse
of~$\lambda$. It can be explicitly described as follows.

By Lemma~\ref{lem:iso} we have
\begin{equation}\label{eq:iso}
\Om^\bullet_{A_f}/\Om^\bullet_{A}\simeq \hat{\Om}^\bullet_{\hat{A}_f}/\hat{\Om}^\bullet_{\hat{A}}\simeq
\hat{\Om}^\bullet_{B[[T]][T^{-1}]}/\hat{\Om}^\bullet_{B[[T]]}.
\end{equation}
To construct the image $\Res(\om)$ of $\om\in\hat{\Om}^p_{B[[T]][T^{-1}]}$, expand
it in powers of $T$ and extract the coefficient of $\frac{\ud T}{T}$.


By linearity, it suffices to consider terms of the form
$$
\om=\frac{a}{f^s}\ud X_{i_1}\wedge\cdots\wedge\ud X_{i_p},\quad a\in A,\ 
i_1<\cdots<i_p,\ s\ge 1.
$$

To construct the image of $\om$ in $\hat{\Om}^\bullet_{B[[T]][T^{-1}]}/\hat{\Om}^\bullet_{B[[T]]}$
under the isomorphism~\eqref{eq:iso}, first note that by Lemma~\ref{lem:iso}
we have
\begin{align*}
b:= & \hat{\psi}^{-1}(a)=\sum_\nu b_{\nu}T^\nu\in B[[T]],\\
\Xi_i:= & \hat{\psi}^{-1}(X_i)=\sum_\nu b_{i\nu}T^\nu\in B[[T]],
\end{align*}
where
\begin{equation}\label{eq:boundCoeff}
\deg b_{\nu}\le (2d^{n}+d)^\nu\deg a,\quad \deg b_{i\nu}\le (2d^{n}+d)^\nu.
\end{equation}
Hence, $\om$ is mapped by the isomorphism~\eqref{eq:iso} to
$$
\hat{\om}=\frac{b}{T^s}\ud \Xi_{i_1}\wedge\cdots\wedge\ud \Xi_{i_p}.
$$
As noted in \S\ref{sec:basics}, we have
\begin{align*}
\ud \Xi_i & = \sum_{\nu\ge 0} (\nu b_{i\nu} T^{\nu-1}\ud T +\ud_B b_{i\nu}T^\nu)\\
 & = \sum_{\nu\ge 0} ( (\nu+1) b_{i,\nu+1} \ud T +\ud_B b_{i\nu})T^\nu.
\end{align*}
The terms of $\hat{\om}$ involving $\ud T$ are of the form
\begin{multline*}
\pm(\nu_1+1) b_\mu b_{i,\nu_1+1} T^{\mu+\nu_1+\cdots+\nu_p-s} \\
\cdot\ud T\wedge\ud_B b_{j_1,\nu_2}\wedge\cdots\wedge\ud_B b_{j_{p-1},\nu_{p}}
\end{multline*}
with some $1\le i,j_1,\ldots,j_{p-1}\le n$ and $\mu,\nu_1,\ldots,\nu_p\ge 0$.
To get the coefficient of $\ud T/T$, we have to consider the case
$\mu+\nu_1+\cdots+\nu_p=s-1$. Using that $\ud_B$ is of degree $0$ together with the 
estimate~\eqref{eq:boundCoeff}, it follows that this
coefficient is of degree
\begin{align*}
\le & \deg b_\mu+\deg b_{i,\nu_1+1}+\deg b_{j_1,\nu_2}+\cdots+\deg b_{j_{p-1},\nu_{p}} \\
\le & (2d^{n}+d)^s(\deg a+p)=(2d^{n}+d)^s(\deg \om+sd)
\end{align*}
which concludes the proof of the theorem.
\end{proof}

\begin{example}[cont.]
We keep the notation of Example~\ref{ex:univar}.
To confirm the proposed action of the residue map, consider
$\om=\frac{g}{f^s}\ud X\in\Om_{\C[X]_f}$. 
We claim that
$$
\Res(\om)=\sum_{i=1}^d\Res_{\z_i}\big(\frac{g}{f^s}\big)e_i,
$$
where $\Res_{\z_i}$ denotes the classical \textit{residue} at $\z_i$ of a
meromorphic function. Recall that $\Res_{\z_i}(h)$ is the coefficient of
$(X-\z_i)^{-1}$ in the Laurent expansion of $h$ around $\z_i$.

According to the proof of Theorem~\ref{thm:effGysin}, we have to check that
$$
\Res(\om)\cdot\frac{\ud f}{f}\equiv\om\pmod{\ud\,\C[X]_f}.
$$
This follows easily from the formulas
\begin{itemize}
\item $\ud f=\sum_i a_i e_i\ud X$,
\item $e_i e_j\equiv\d_{ij}e_i\pmod{(f)}$,
\item $\frac{a_i e_i}{f}=\frac{1}{X-\z_i}$,
\item $\om\equiv\sum_i \Res_{\z_i}(\frac{g}{f^s})\frac{1}{X-\z_i}\ud X\pmod{\ud\C[X]_f}$.
\end{itemize}
The last identity follows again from the partial fraction decomposition.
\end{example}

Now we are in the position to prove our main result.
\begin{theorem}\label{thm:effDeRham}
For each smooth hypersurface $X\subseteq \C^n$ of degree $d\ge 3$ we have
\begin{align*}
\deg(\Hdr^p(X)) & \le (p+1)(d+1)(2d^n+d)^{p+1} \\
 & \le d^{\Oh(pn)}.
\end{align*}
\end{theorem}
\begin{proof}
Let $X=\mZ(f)$, where $f$ is squarefree of degree $d$.
If we denote by $\tilde{f}$ the generous homogenization $X_0^{d+1}f(X/X_0)$,
then we have
$U:=\C^n\setminus \mZ(f)=\proj^n\setminus\mZ(\tilde{f})$.
As stated in \S\ref{ss:knownCases},
each cohomology class in $\Hdr^{p+1}(U)$ is represented by a differential form
$$
\frac{\tilde{\a}}{\tilde{f}^{p+1}},\quad\deg\tilde{\a}=(p+1)(d+1).
$$
Dehomogenizing yields a form $\om=\a/f^{p+1}$ with $\deg\a\le (p+1)(d+1)$, hence
$\deg\om\le p+1$. Since the residue map is surjective, the bound of
Theorem~\ref{thm:effGysin} implies the claim.
\end{proof}

\begin{example}[cont.]
We have seen that in the univariate case $\Hdr^0(\mZ(f))$ is generated by the
idempotents $e_i$, which are of degree $d-1$. Theorem~\ref{thm:effDeRham} gives
the bound $2d(d+1)$ in this case.
\end{example}

\begin{example}
Consider the hypersurface 
$$
V=\mZ(f)\subseteq\C^2,\quad\text{where}\quad f:=XY^2-X-1.
$$
One easily checks that
$V$ is smooth, but has a singularity at infinity, namely $(u:x:y)=(0:1:0)$.
There is one other point at infinity $(u:x:y)=(0:0:1)$, which is smooth.
Topologically, the projective closure $\overline{V}$ is a sphere with two
points collapsed (to the singularity), so $V$ is a sphere with three points
deleted. It follows that the cohomology is
$$
\Hdr^0(V)=\C\cdot 1,\quad \Hdr^2(V)=0,\quad\text{and}\quad \dim\Hdr^1(V)=2.
$$
Let us find generators of $\Hdr^1(V)$. Note that
\begin{equation}\label{eq:exampleIso}
V\to U:=\C\setminus\{\pm 1\},\quad (x,y)\mapsto y
\end{equation}
is an isomorphism, and $\Hdr^1(U)$ is generated by
$$
\frac{\ud Y}{Y-1},\quad \frac{\ud Y}{Y+1}.
$$
The isomorphism~\eqref{eq:exampleIso} identifies $X$ with $\frac{1}{Y^2-1}$ and
$\ud X$ with
$-\frac{2Y}{Y^2-1}\ud Y$. Hence, $\Hdr^1(V)$ is generated by
$$
X(Y+1)\ud Y,\quad X(Y-1)\ud Y.
$$
Theorem~\ref{thm:effDeRham} gives a bound of $3528$ for the degrees of
generators in this case, so there seems to be room to optimize.

Finally, let us determine the action of the residue map. Its inverse image maps
the generators of $\Hdr^1(V)$ to
$$
\frac{X(Y+1)^2(Y-1)}{f}\ud X\wedge\ud Y,\quad
\frac{X(Y+1)(Y-1)^2}{f}\ud X\wedge\ud Y,
$$
which are cohomologous to
$$
\frac{Y+1}{f}\ud X\wedge\ud Y,\quad
\frac{Y-1}{f}\ud X\wedge\ud Y,
$$
and generate $\Hdr^2(\C^2\setminus V)$.

To find the action for general inputs, note that
$$
X\partial_X f-f=1.
$$
Hence, any 2-form
$$
\frac{h}{f}\ud X\wedge\ud Y=\frac{h}{f}X\partial_X f\ud X\wedge\ud Y
-h\ud X\wedge\ud Y
$$
is equivalent to
$$
\frac{h}{f}X\partial_X f\ud X\wedge\ud Y=
Xh\frac{\ud f}{f}\wedge\ud Y
$$
modulo exact forms, and hence is mapped by the residue map to
$$
Xh\ud Y.
$$
We remark that, also in the general case, one can determine by this method the
image under the residue map for forms of order 1 along $f$ without using the
completion, but this method does not work for higher orders. This also explains the
overestimate of our bound in this example.
\end{example}

\section*{Acknowledgements}
The author thanks Sergei Yakovenko for asking the question addressed in this
paper and bringing to his attention the solution of the infinitesimal Hilbert
16th problem~\cite{bny:10}.
He is also grateful to the Hausdorff Center for Mathematics, Bonn, for its
kind support. 


\begin{thebibliography}{10}

\bibitem{basu:06}
S.~Basu.
\newblock Computing the first few {B}etti numbers of semi-algebraic sets in
  single exponential time.
\newblock {\em J. Symbolic Comput.}, 41(10):1125--1154, 2006.

\bibitem{basu:08}
S.~Basu.
\newblock Algorithmic semi-algebraic geometry and topology -- recent progress
  and open problems.
\newblock In {\em Surveys on discrete and computational geometry}, volume 453
  of {\em Contemp. Math.}, pages 139--212. Amer. Math. Soc., Providence, RI,
  2008.

\bibitem{bpr:03}
S.~Basu, R.~Pollack, and M.-F. Roy.
\newblock {\em Algorithms in {R}eal {A}lgebraic {G}eometry}, volume~10 of {\em
  {A}lgorithms and {C}omputation in {M}athematics}.
\newblock Springer-Verlag, Berlin Heidelberg New York, 2003.

\bibitem{bgmw:11}
E.~Bierstone, D.~Grigoriev, P.~Milman, and J.~Wlodarczyk.
\newblock Effective {H}ironaka resolution and its complexity.
\newblock {\em Asian J.~Math.}, 15(2):193--228, 2011.

\bibitem{bny:10}
G.~Binyamini, D.~Novikov, and S.~Yakovenko.
\newblock On the number of zeros of abelian integrals.
\newblock {\em Invent.\ Math.}, 181:227--289, 2010.

\bibitem{brow:87}
W.~Brownawell.
\newblock Bounds for the degrees in the {N}ullstellensatz.
\newblock {\em Ann. of Math. (2)}, 126(3):577--591, 1987.

\bibitem{bus:09}
P.~B{\"u}rgisser and P.~Scheiblechner.
\newblock On the complexity of counting components of algebraic varieties.
\newblock {\em J.~Symb.\ Comp.}, 44(9):1114--1136, 2009.
\newblock Effective Methods in Algebraic Geometry.

\bibitem{ded:90}
P.~Deligne and A.~Dimca.
\newblock Filtrations de {H}odge et par l'ordre du p\^ole pour les
  hypersurfaces singuli\`eres.
\newblock {\em Annales scientifiques de l'\'Ecole Normale Sup\'erieure},
  23(4):645--656, 1990.

\bibitem{dimc:90}
A.~Dimca.
\newblock On the {M}ilnor fibrations of weighted homogeneous polynomials.
\newblock {\em Compositio Mathematica}, 76(1--2):19--47, 1990.

\bibitem{dimc:92}
A.~Dimca.
\newblock {\em Singularities and Topology of Hypersurfaces}.
\newblock Universitext. Springer Verlag, 1992.

\bibitem{eise:95}
D.~Eisenbud.
\newblock {\em Commutative {A}lgebra with a {V}iew {T}oward {A}lgebraic
  {G}eometry}, volume 150 of {\em Graduate Texts in Mathematics}.
\newblock Springer-Verlag, New York, 1995.

\bibitem{gro:64}
A.~Grothendieck.
\newblock \'{E}l\'ements de g\'eom\'etrie alg\'ebrique. {IV}. \'{E}tude locale
  des sch\'emas et des morphismes de sch\'emas. {I}.
\newblock {\em Inst. Hautes \'Etudes Sci. Publ. Math.}, (20):259, 1964.

\bibitem{gro:66}
A.~Grothendieck.
\newblock On the de {R}ham cohomology of algebraic varieties.
\newblock {\em Publications Math{\'e}matiques {IHES}}, 39:93--103, 1966.

\bibitem{ham:83}
H.~Hamm.
\newblock Lefschetz theorems for singular varieties.
\newblock In {\em Singularities, {P}art 1 ({A}rcata, {C}alif., 1981)},
  volume~40 of {\em Proc. Sympos. Pure Math.}, pages 547--557. Amer. Math.
  Soc., Providence, RI, 1983.

\bibitem{hal:85}
H.~Hamm and D.~T. L{\^e}.
\newblock Lefschetz theorems on quasiprojective varieties.
\newblock {\em Bull. Soc. Math. France}, 113(2):123--142, 1985.

\bibitem{har:75}
R.~Hartshorne.
\newblock On the de {R}ham cohomology of algebraic varieties.
\newblock {\em Publications Math{\'e}matiques de L'IH{\'E}S}, 45:6--99, 1975.
\newblock 10.1007/BF02684298.

\bibitem{jel:05}
Z.~Jelonek.
\newblock On the effective {N}ullstellensatz.
\newblock {\em Invent.\ Math.}, 162(1):1--17, 2005.

\bibitem{koll:88}
J.~Koll{\'a}r.
\newblock Sharp effective {N}ullstellensatz.
\newblock {\em J. Amer. Math. Soc.}, 1(4):963--975, 1988.

\bibitem{kun:86}
E.~Kunz.
\newblock {\em K{\"a}hler Differentials}.
\newblock Advanced Lectures in Mathematics. Vieweg, Wiesbaden, 1986.

\bibitem{mon:68}
P.~Monsky.
\newblock Formal cohomology: {II}. {T}he cohomology sequence of a pair.
\newblock {\em Ann. of Math.}, 88(2):218--238, 1968.

\bibitem{ota:98}
T.~Oaku and N.~Takayama.
\newblock An algorithm for de {R}ham cohomology groups of the complement of an
  affine variety via {D}-module computation.
\newblock {\em Journal of Pure and Applied Algebra}, 139:201--233, 1999.

\bibitem{sch:07}
P.~Scheiblechner.
\newblock On the complexity of deciding connectedness and computing {B}etti
  numbers of a complex algebraic variety.
\newblock {\em J.~Compl.}, 23(3):359--379, 2007.

\bibitem{sch:09}
P.~Scheiblechner.
\newblock Castelnuovo-{M}umford regularity and computing the de {R}ham
  cohomology of smooth projective varieties.
\newblock arXiv:0905.2212v3, accepted for Foundations of Computational
  Mathematics, 2009.

\bibitem{sch:10}
P.~Scheiblechner.
\newblock On a generalization of {S}tickelberger's theorem.
\newblock {\em J.~Symb.\ Comp.}, 45(12):1459 -- 1470, 2010.
\newblock MEGA'2009.

\bibitem{wal:00}
U.~Walther.
\newblock Algorithmic computation of de {R}ham cohomology of complements of
  complex affine varieties.
\newblock {\em J.~Symb.\ Comp.}, 29(4-5):795--839, 2000.

\bibitem{wal:00b}
U.~Walther.
\newblock Algorithmic determination of the rational cohomology of complex
  varieties via differential forms.
\newblock In {\em Symbolic computation: solving equations in algebra, geometry,
  and engineering ({S}outh {H}adley, {MA}, 2000)}, volume 286 of {\em Contemp.
  Math.}, pages 185--206. Amer. Math. Soc., Providence, RI, 2001.

\end{thebibliography}

\end{document}